%% file: Kopeliovich-Shaska.tex
\documentclass{amsart}

\include{macro}

\usepackage{graphicx}

\usepackage{mathtools}
\usepackage{xy}\input xy \xyoption{all}

\theoremstyle{definition}
\newtheorem{definition}{Definition}[section]

\def\cL{\mathcal L}

\makeatletter
\newcommand*\bigcdot{\mathpalette\bigcdot@{.7}}
\newcommand*\bigcdot@[2]{\mathbin{\vcenter{\hbox{\scalebox{#2}{$\m@th#1\bullet$}}}}}
\makeatother

\def\x{\mathbf x}
\def\A{\mathbb A}
\def\cO{\mathcal O}

\title{The addition  on Jacobian varieties from a geometric viewpoint}

\author{Yaacov Kopeliovich}

\author{Tony Shaska}



\begin{document}

\begin{abstract}
We give a geometric interpretation of the group law for   Jacobian varieties by extending the geometric construction of chords and tangents on an elliptic curve.
For any given  algebraic  curve  $\X$  and reduced divisors $D_1, D_2 \in \Jac \X$,  we define  curves $\X^\prime$ and $\X^{"}$ such that the intersection  $\X \cap \X^\prime$ determines  precisely 
the divisor  $-(D_1+D_2)$ and the intersection $\X\cap \X^{"}$ determines $D_1+D_2$.  
For superelliptic curves  such formulas are made explicit. 
\end{abstract}

\maketitle


\section{Introduction}
%
Given a genus $g\geq 0$ curve $\X$,  defined over a field $k$, the addition law in $\Jac (\X)$ has been a topic of interest since early XIX century.    The classical definition of the group law on a cubic curve was done via the chord and tangent construction and has become part of the mathematical folklore.   Jacobi took such idea and algebraically generalized it to all curves;  see \cite{jacobi-1846}. Further work of   
\cite{W},  \cite{Brill-80}, \cite{Noether-84}, \cite{salmon},      among others laid the foundation for the concept of Jacobian variety as we understand it today.   Noether's Fundamental theorem, considered as the starting point of intersection theory, gives nice applications of defining the group structure geometrically on conics and cubics (cf. \cref{section-2}). 

Generalizing this concept to higher genus required  a deeper understanding  of the function field $k(\X)/k$ and rational functions with a given set of points as zeroes or poles,  which led to the celebrated  Noether Gap Theorem and the concepts of adjoint polynomials and $\phi$-polynomials;  see \cite[Chap. VI]{baker}, \cite{coolidge} for a survey of classical results or \cite{fulton} for a more modern approach. 
However, it was not until  a decade ago that a geometric interpretation of such addition  was given in  \cite{Leitenberger} for hyperelliptic curves, where well known ideas used in \cite{cantor-1} were interpreted interpreted geometrically.   What makes hyperelliptic curves so easily accessible  is the fact that  they are degree two coverings of  the projective line, or equivalently the rational functions on such curves have poles of order .....  . Hence, it is no surprise that throughout classical works, hyperelliptic curves have been widely used to give the first examples of the general theory;  \cite{hancock},   \cite{Mum-1}, \cite{tata-2}, \cite{buchstaber}.


In this paper, we introduce a geometric approach on defining the group law on Jacobian varieties based on the intersection theory of curves.   Our approach is as follows: 

Fix a point $P\in \X$ and a polarization 
\[\iota_P :  \X \hookrightarrow \Jac (\X).\]
Determine a  basis $\B$   for $k(\X)/k$ and order such elements of the basis based on their order at $P$, we call this the \emph{adopted basis } at $P$.  
Any divisor $D\in \Jac (\X)$ can be written as a reduced divisor $D= \sum_{i=1}^g P_i - g P$, where $P_i\in \X$, for $i=1, \ldots , g$ (not necessarily distinct).   Without loss of generality we can take $P$ to be the point at infinity.

For any two divisors $D_1=\sum_{i=1}^g P_i - g \infty$ and $D_2=\sum_{i=g+1}^{2g} P_i - g\infty$,  we take the first $2g+1$ elements  of $\B$ and define an \emph{interpolation curve} $\X^\prime$ which passes though the points $P_1, \ldots , P_{2g}$ and intersects $\X$ in $P_{2g+1}, \ldots , P_{3g}$ new points.    Let $D^\prime= \sum_{2g+1}^{3g } - g \infty$. Then 
Then $D_1+D_2 = - D$.    To invert $D$ we use the points $P_{2g+1}, \dots , P_{3g}$ and the set of   the first $g+1$ elements of $\B$ and define another interpolation curve $\X^{"}$ which passes through $P_{2g+1}, \ldots , P_{3g}$ and intersects $\X$ exactly in $g$ new points $P_{3g+1}, \ldots , P_{4g}$.  We show in  \cref{main-thm} that $D_1 + D_2 = \sum_{i=3g+1}^{4g} P_i - g \infty$.

Hence, when the sets of points $D_1=(P_1, \ldots , P_g)$ and $D_2=(P_{g+1}, \ldots , P_{2g})$ move through $\X^g$, we have a set of moving curves $\X^\prime$ and $\X^{"}$. 
It is unknown to us what was  exactly known to classical algebraic geometers about such curves or if they have ever been defined before. 
Thus, for a given curve $\X$   and any two sets of generic cycles $D_1=(P_1, \ldots , P_g) \in \X^g$,  $D_2=(P_{g+1}, \ldots , P_{2g}) \in \X^g$, there exists a curve $\X^{"}$   which passes through $P_1, \ldots , P_{2g}$ and intersects $\X$ transversally at exactly $3g$ points, which seems  similar to  the  analogue of a moving lemma for cycles $A:=D_1 + D_2$ and  $B:= \X \cap \X^\prime$.

Hence, similarly  for cubics, we define a map 
\[
\begin{split}
\varphi :  \X^g  \times \X^g  & \to \X^g \\ 
\left( (P_1, \ldots , P_g),    (P_{g+1}, \ldots , P_{2g})   \frac {}{}  \right)  & \to  ( Q_1, \ldots , Q_g),
\end{split}
\]
which can be extended to a map between Jacobians
\[ \bar \varphi : \; \Jac \X \times \Jac \X \to \Jac \X
\]
and gives the addition law in $\Jac \X$.

The formulas for the addition can be determined explicitly as long as we are able to determine an adopted basis $\B$.  Such bases are known for cyclic  curves (or superelliptic under mild restrictions). In  \cref{superelliptic} we explore such curves in detail and give explicit adopted ordered bases for each genus $g\geq 2$.  
For such curves we are able to determine explicitly the first $2g+1$ monomials of the adopted basis $\B$  for $P=\infty$   (cf. \cref{pattern}) and therefore provide explicit formulas for the addition law in the Jacobian $\Jac \X$.
In the case when $n=2$, hence the curve is hyperelliptic, our formulas agree with well known formulas for hyperelliptic curves;  see \cite{frey-shaska} for details.

The important fact in our approach is the definition of the interpolating curve $\X$ through the adopted basis $\B$ to $P$ (cf. \cref{matrix-A}). 
The concept of adopted basis at a point $P\in \X$ seems to be new in the literature.  
There are similar attempts to define the matrix $A$ which determines the interpolating curves through a different set of polynomials, the so called adjoint integral polynomials, or $\phi$-polynomials; see \cite[pg. 147]{baker}.  
Our approach seem to fully explain geometrically the addition law on Jacobian varieties. It remains to be seen if it ca be interpreted through the classical work of Noether, von Brill, and others.

\section{Intersections of curves}\label{section-2}
Let $k$ be a perfect field and $\P^n (k)$ the usual projective space. When there is no confusion we will simply denote it by $\P^n$. 

A zeo-cycle in $\P^2$ is a formal sum $\sum_{P\in \P^2} n_p P$, where $n_P$'s are integers and almost all of them are zero.  The set of zero-cycles is the free Abelian group with basis $\P^2$.  The degree of a zero cycle 
$\sum n_p P$ is $\sum n_p $.  

Let $\X_1$ and $\X_2$ be any two projective, irreducible curves of degree $m$ and $n$.   The \textbf{intersection cycle} $\X_1 \bigcdot \X_2$  is defined as 
\[   \X_1 \bigcdot \X_2 = \sum_{P\in \P^2} (\X_1 \cap \X_2)_P \, P, \]
where $(\X_1 \cap \X_2)_P$ is the intersection index of $(\X_1 \cap \X_2)$  at $P$.  From Bezout's theorem, $\X_1 \bigcdot  \X_2$ is a positive zero-cycle of degree $mn$.  
Recall that two curves  $\X_1$ and $\X_2$ \textbf{intersect transversally} at $P$ if  $P$ is a simple point at both $\X_1$ and $\X_2$ and if the tangents to $\X_1$ and $\X_2$ are different at $P$. 


\begin{thm}[Noether Fundamental Theorem]\label{noether}
Let $\X_1, \X_2, \X_3$ be three projective plane curves with equations
\[
\X_1:  f(x, y, z)=0, \quad \X_2: g(x, y, z) =0, \quad  \X_3 : h(x, y, z) =0,
\]
 such that $\X_1$ and $\X_2$ have no common components. There exists $A(x, y, z)$ and $B(x, y, z)$ such that 
\[  h(x,y,z)  =  A(x, y,z)  f(x, y,z) + B(x, y,z) g(x, y,z) , \]
with  $\deg A = \deg g - \deg f$ and $\deg B = \deg h - \deg g$   if and only if one of the following is  satisfied at every $P \in \X_1 \cap \X_2$.

i) $\X_1$ and $\X_2$ meet transversally at $P$ and $P\in \X_3$.

ii) $P$ is simple on $\X_1$ and $(\X_1 \cap \X_3)_P \geq (\X_1 \cap \X_2)_P$

iii) $\X_1$ and $\X_2$ have distinct tangents at $P$, and 
\[\mult_P (\X_3) \geq \mult_P (\X_1) + \mult_P (\X_2) -1.\]
\end{thm}
See \cite[pg. 61]{fulton} for details.    
There are many applications in geometry of Noether's fundamental theorem, but we are especially interested in the following. 
 
\begin{exa}[Addition on a cubic]
Let $\X$ be a smooth cubic defined over a field $k$ (not necessarily closed).   Fix $\cO\in \X(k)$.
For two points $P, Q \in \X (k)$ there is a unique line $\X_1$ such that $\X \bigcdot \X_1 = P+Q+R$, for some $R\in \X$.   Define   the function
\[
\begin{split}
 \varphi: \quad  \X \times \X  &  \to \X \\
 (P, Q) & \to R \\
\end{split}
\]
%

For the two points $\cO$ and $R$ there is another line $\X_2$ such that $\X \bigcdot \X_2 = \cO + R +S$, for some $S\in \X$. 
 We define an addition in $\X$ such that 
\[
P \oplus Q = \varphi (\cO, \varphi(P, Q))
\]
Then $(\X (k), \oplus)$ is an Abelian group. 
\end{exa} 
The goal of this paper is to investigate whether  the construction above can be generalized for any genus $g\geq 1$ such that the addition defined for cubics can carry to the addition on Jacobian of the curve. We want to  get a better understanding of the curves $\X_2$ and $\X_3$ and  if possible, we would like to determine the equations of $\X_2$ and $\X_3$ explicitly when the equation of $\X_1$ is given.

%
Let $P_1,P_2, \ldots ,$ be a sequence of (not necessarily distinct) points on $\X_g$.  Let define the following
\[
 D_0  = 0 \quad \text{ and } D_k = P_1 + \dots + P_k. \]
One can ask the following question:   \textit{For each nonnegative $k$, does there exist e meromorphic function $f$ on $\X_g$ whose polar divisor $(f)_\infty $ satisfies $(f)_\infty \leq D_k$ and $(f)_\infty \not\leq D_{k-1}$?} 
If the answer to the above question for a given $k$ is "No" then we say that $k$ is a \textbf{Noether gap} for the sequence $P_1, P_2, \dots $, otherwise is a \textbf{non-gap}.

\begin{thm}[Noether gap theorem] 
For any sequence $P_1,P_2, \dots$, there are exactly $g$ Noether gap numbers $n_i$ with \[1=n_1<n_2<\dots<n_g\leq 2g-1.\]
\end{thm}

The Weierstrass ``gap'' theorem is a special case of the Noether ``gap'' theorem, taking $P_i=P$ for all $i$.
%
%
Fix a point $P\in \X$ such that $P$ is a Weierstrass point.         We have the following theorem from \cite{Leykin}

\begin{thm}\label{thm-1}
Any generic collection of   points $P_1,\dots , P_{g+s} \in \X$, where $s\geq 0$,     can be realized uniquely as zeros of a meromorphic function $\Phi (x, y)$ or order at most $2g+s$ and this function is unique up to multiplication by a constant.
\end{thm} 

\begin{proof} 
A meromorphic function $\Phi(x, y)$ belongs to the function field    $k(\X)$.   
We can consider a basis  of $k(\X)$ at a Weierstrass point $P \in \X$.    By the Weierstrass gap theorem for a function of order $2g+s$ we will have at most $g+s$ orders at $P$ (as there are no  functions at the gaps) and hence this function will be determined uniquely by the $g+s$ points $P_1,\dots , P_{g+s}$.
\end{proof}


\begin{rem}
Let $m$ be the first (i.e. the smallest) non-gap at $P$ and $n$ the next relatively prime non-gap.   Denote by $x$ and $y$ two meromorphic functions on $\X$ with a pole at $P$ of order $m$ and $n$, respectively, as their only singularity. Then, Weierstrass showed in \cite[pg. 297-307]{W}, that  $x$ and $y$ satisfy the irreducible equation
\begin{equation}\label{W-normal-form}
 y^m + A_1 (x) y^{m-1} + \cdots + A_{m-1} (x) y + A_m(x) =0, 
\end{equation}
where $A_i (x) \in k[x]$ with $\deg A_i (x) \leq \frac {n_i} m$, for $i < m$,    $\deg A_m (x) =n$. Moreover,   $\X$ is defined by this equation.   This is the so called \textbf{Weierstrass normal form}  of $\X$.   For a proof of this statement see \cite{baker}.     A nice historical account of Weierstrass points is given in \cite{del-centina}. 
\end{rem}

%
Fix $P \in \X$ and  let $(x_P, y_P )$ be a local coordinate around $P$
By a   \textbf{basis  adapted to $P$},  we mean  a basis 
\[\B:= \{ 1,  \varphi_1 ,  \varphi_2 ,  \cdots ,  \varphi_m, \ldots    \},\] 
of the of the function field $k(\X)/k$ ordered  according to their order at $P$, 
\[\ord_P (\varphi_1) < \ldots < \ord_P (\varphi_i) < \ldots < \ord_P (\varphi_m).\]
Such basis $\B$  adapted to $P$  is not unique.   
%
To make it unique consider the Taylor series expansion of $\varphi_i$ at $P$, say 
\[
\varphi_i (t) = \sum_{j=0}^\infty a_{i, j} (t-P)^j
\]
and require that $a_{i, j}=1$ for $i=j$ and $a_{i, j}=0$ otherwise.   The \textbf{weight of $P$ with respect to $\B$} is defined as 
\[
\tau (P) = \sum_{i=1}^m \left(  \ord_P (\varphi_i  ) - i+1   \right)
\]
%
%
If $P$ is the place at infinity, then we can assume that $\B$ is a monomial basis. In this case $\B$ is unique.  

Our main theorem is the following:
\begin{thm}\label{main-thm}
Let $\X: F_1 (x, y, z)=0$ be a smooth, projective, genus $g\geq 1$ curve defined over $k$,   $P\in \X (k)$, and $\B= \{   1,  \varphi_1 ,  \varphi_2 ,  \cdots ,  \varphi_m, \ldots       \}$ a
basis adapted to $P$.
For any generic set of points $P_1, \ldots , P_m \in \X (k)$, for $m \leq 2g$, there exist unique curves $\X^\prime$ and $\X^{"}$ such that:
\begin{itemize}
\item[i)] $\X^\prime: \; F_2 (x, y, z)=0$ is a degree $d_1=\deg (\varphi_{m+g})$ curve which meets $\X$  transversally at $(m+g)$ points, say   
\[ \X \bigcdot \X^\prime = \sum_{i=1}^{m+g} P_i.\]
  Then $\deg \X \bigcdot \X^\prime \leq m+g$.
%

\item[ii)] $\X^{"}: \; F_3 (x, y, z)=0$   is a degree $d_2=\deg \varphi_g$  curve which    meets $\X$  transversally at $m+2g$ points  and 
\[\X \bigcdot \X^{"} =  \left( \sum_{i=m+1}^{m+g} P_i \right)   + \sum_{i=1}^g Q_i, \] 
for some $Q_1, \dots , Q_g \in \X$.  

\item[iii)] There exists polynomials $A, B \in k[x, y, z]$ such that 
\begin{equation}\label{eq-noether}
  F_3 (x,y,z)  =    F_1 (x, y, z) \, A(x, y, z) +  F_2(x, y, z) \, B(x, y, z) , 
\end{equation}
with  $\deg A = \deg F_2 - \deg F_1$ and $\deg B = \deg F_3 - \deg F_2$ . 

\item[iv)] If $m=2g$ then the sum of the zero-cycles $D_1=\sum_{i=1}^g P_i$ and $D_2 =\sum_{j=g+1}^{2g} P_j$ is given by the formula
$ D_1 + D_2= \sum_{i=1}^g Q_i$. 
\end{itemize}
\end{thm}

\proof    For a point $P \in \X$,  let $\B$ be a basis  $\B$ adapted to $P$. 
Given  points $P_1, P_2, \dots , P_m \in \X$,  we take the first $(m+1)$ functions $\varphi_1, \dots , \varphi_{m+1}$ of $\B$ (i.e., the ones with smallest order at $P$).   Define  the interpolating matrix $A$ as 
\begin{equation}\label{matrix-A}
A_P  (P_1, \dots P_m):= 
\begin{bmatrix}
\varphi_1 (x, y) & \varphi_2 (x, y)  & \dots & \varphi_{m+1} (x, y) \\
\varphi_1 (x_1, y_1) & \varphi_2 (x_1, y_1) & \dots & \varphi_{m+1} (x_1, y_1) \\
\vdots & \vdots & \dots & \vdots \\
\varphi_1 (x_m, y_m) & \varphi_2 (x_m, y_m)  & \dots & \varphi_{m+1} (x_m, y_m)  \\
\end{bmatrix} 
\end{equation}
which depends only on the base point $P\in \X$ and the zero-cycle $D=\sum_{i=1}^m P_i$. 

 Let   $\X^\prime$ be the curve defined by 
\begin{equation}\label{eq-Y}
\X^\prime: \; \;  \det A_P  (P_1, \dots P_m) = 0.
\end{equation}
To show that $P_i \in \X^\prime$ for $i=1, \dots , m$ it is enough to show that when we substitute $(x, y)$ by $(x_i, y_i)$ in \cref{matrix-A}, then  $\det A_P  (P_1, \dots P_m)=0$.   But this is obvious since in this case the matrix $A_P$ has two identical rows. 

Consider the determinant $\det A$. The coefficient of  the   $\varphi_i$ is $(-1)^{1+j}B_{1j}$, where $B_{1j}$ is the minor obtained by removing the $1$-st row and the $j$-th column. Recall that the poles of $\varphi_1, \dots , \varphi_m$ have at most order $g$.   
Thus we can view the $\det A$ as a polynomial in $x$ and $y$ of degree $m+g$, since by clearing out denominators we can only have degree $g$ monomials.

The intersection cycle  $\X \bigcdot \X^\prime$  is principal and generated by the monomials of $\varphi_1, \dots , \varphi_m$.  Since all the monomials have degree $\leq  m+g$, this divisor will have degree $\leq m+g$.  
 %
This completes the proof of  i). 


To prove part ii) we start with the points $P_{m+1}, \ldots , P_{m+g} \in \X$ and apply part i) to these points.  Hence we have a new curve $\X^{"}$ such that it intersects $\X$ in exactly $m+2g$ points, from which $P_{m+1} , \ldots , P_{m+g}$ are already points of intersection.   Denote the new points of intersection by $Q_1, \dots , Q_g$. 
 Then $\X\bigcdot \X^{"}$ as claimed.

Part iii) follows from the \cref{noether}.   Take curves $\X_1, \X_2, \X_3$ as $\X, \X^\prime$, and $\X^{"}$ respectfully.   Since $\X$ and $\X^\prime$ meet transversally at all $P \in \X \cap \X^\prime$ then conditions of the Noether's theorem are satisfied.  Hence, exist $A, B\in k[x, y, z]$ such that \cref{eq-noether} is satisfied.  

Let $D_1$ and $D_2$ as in the hypothesis of part iv).   From \cref{noether},   $\X_1 \cap \X_2$ is a principal divisor.  Hence, $D_1 + D_2 = - \sum_{i=2g+1}^{3g} P_i$.  By the same argument, since $\X^\prime \cap \X^{"}$ is a principal divisor then $- \sum_{i=2g+1}^{3g} P_i = \sum_{i=1}^g Q_i$.  This concludes the proof. 
\qed

   As the sets of points $D_1=(P_1, \ldots , P_g)$ and $D_2=(P_{g+1}, \ldots , P_{2g})$ move through $\X$, we have a family of \textbf{moving curves} $\X^\prime$ and $\X^{"}$.   We are especially interested in $\X^{"}$. 
   
\begin{cor}
For a given curve $\X$   and two sets of generic cycles $D_1=(P_1, \ldots , P_g) \in \X^g$,  $D_2=(P_{g+1}, \ldots , P_{2g}) \in \X^g$, there exists a curve $\X^{"}$   which passes through $P_1, \ldots , P_{2g}$ and intersects $\X$ transversally at exactly $3g$ points. 
\end{cor}

The corollary above looks as an analogue of a moving lemma for cycles $A:=D_1 + D_2$ and  $B:= \X \cap \X^\prime$.  Such ideas  trace back to Severi \cite{severi} as pointed out by Lazarsfeld in \cite{lazarsfeld}. 
Next we define the map 
\begin{equation}\label{add-map}
\begin{split}
\varphi :  \X^g  \times \X^g  & \to \X^g \\ 
\left( (P_1, \ldots , P_g),    (P_{g+1}, \ldots , P_{2g})   \frac {}{}  \right)  & \to  ( Q_1, \ldots , Q_g)
\end{split}
\end{equation}
which we call it \textbf{the addition map}.   In the next sections we will see how this map gives us the addition on  $\Jac \X$.   Determining the degree, genus, equation, and the minimal field of definition of the curve $\X^{"}$ would determine explicitly the addition on $\Jac \X$.  Hence, in the next section we study in more details the interpolating curves. 

\begin{rem}
For a given general curve $\X$, say with equation as in \cref{W-normal-form}, it is not known how to determine  an adopted  basis $\B$ at $P \in \X$.  However we can do this for cyclic curves, namely curves $\X$ which have a cyclic covering $\X \to \P^1$. We will explain this in more detail in the coming sections. 
\end{rem}

\begin{rem}
The concept of \emph{adopted basis} as used above seems to be new in the literature.  The closest to it seems the set of adjoint polynomials as described in  \cite[pg. 124-127]{baker} or more specifically of  $\phi$-polynomials \cite[pg. 141-143]{baker}.  A  matrix  similar to \cref{matrix-A} is given in terms of such $\phi$-polynomials in \cite[pg. 147]{baker}.
\end{rem}

\section{Interpolation}

%
The curves $\X^\prime$ and $\X^{"}$ introduced in the previous section are  \textbf{interpolating curves} with respect to a given corresponding basis $\B$.   We will explore this idea to further detail in this section.

We consider first a more general situation when we have an ideal  $I \subset k[x, y]$.  We order the variables as $x > y$.  For $\a=(\a_1, \a_2) \in \Z_{\geq 0}^2$ we denote by $\x^\a$ the monomial $x^{\a_1} y^{\a_2}$.  
Fix a monomial ordering and let  $\< LT (I) \>$ be the ideal generated by the leading terms of the polynomials in $I$. There are $\varphi_1, \ldots , \varphi_m \in I$ such that $\<LT (I)\> = \<LT (\varphi_1), \ldots , LT (\varphi_m) \>$.  Hence  $G= \{ \varphi_1, \ldots , \varphi_m\} $ is a Groebner basis for $I$.  We can even assume that this is a minimal Groebner basis.   
Let 
\[ \B:= LT (\varphi_1), \ldots , LT (\varphi_m) .\]
Consider the following problem: 
\emph{Given points $P_i (x_i, y_i) \in \A^2 (k)$,  for $i=1, \ldots , m$, find a polynomial $f \in \< LT (I)\>$,   which vanishes on all points $P_i (x_i, y_i) \in \A^2 (k)$,  for $i=1, \ldots , m$. }

Hence, the solution $f(x, y)$ will be a $k$-linear combination of  elements of $\B$.  The answer is $f(x, y) : = \det A$, where $A$ is defined as in \cref{matrix-A}.  More explicitly, if $\a_1, \ldots , \a_m \in \Z_{\geq 0}^m$ such that $\a_i = ( \a_{i, 1}, \a_{i, 2})$,  take 
$\B = \{ \x^{\a_1}, \ldots , \x^{\a_m}\}$, where $\x^{\a_i} = x^{\a_{i, 1}} y^{\a_{i, 2}}$,  and then the matrix $A$ is given by
\begin{equation}\label{matrix-AA}
A=
\begin{bmatrix}
1 & x^{\a_{1, 1}} y^{\a_{1, 2}}     & x^{\a_{2, 1}} y^{\a_{2, 2}}   & \dots &  x^{\a_{m, 1}} y^{\a_{m, 2}}  \\
1& x_1^{\a_{1, 1}} y_1^{\a_{1, 2}} &  x_1^{\a_{2, 1}} y_1^{\a_{2, 2}}    & \dots &  x_1^{\a_{m, 1}} y_1^{\a_{m, 2}}  \\
\vdots & \vdots & \vdots & \dots & \vdots \\
1 & x_m^{\a_{1, 1}} y_m^{\a_{1, 2}}&  x_m^{\a_{2, 1}} y_m^{\a_{2, 2}}    & \dots & x_m^{\a_{m, 1}} y_m^{\a_{m, 2}}    \\
\end{bmatrix} 
\end{equation}
Using Bezutians and subresultants the determinant $\det A= f(x, y)$ can be expressed as a polynomial $f(x, y)$ with coefficients in terms of symmetric polynomials of $x_1, \ldots , x_m$, $y_1, \ldots , y_m$. 
We call the curve $\Y_m : \det A=0$ the interpolating curve of the set of points $\{ P_1, \ldots , P_m\}$ and the monomial basis $\B$.

\bigskip

Let us now come back to our initial problem of the previous section.   For a point $P \in \X$,  let $\B$ be a basis  $\B$ adapted to $P$. 
Given  points $P_1, P_2, \dots , P_m \in \X$,  we take the first $(m+1)$ functions $\varphi_1, \dots , \varphi_{m+1}$ of $\B$ (i.e., the ones with smallest order at $P$).   Define  the interpolating matrix $A$ as in \cref{matrix-A}.  Let the interpolating curve $\Y_m$ be defined by   $ \Y_m : \det A =0$.

\begin{lem}\label{rem-3}
The interpolating curve $\Y_m$   does not depend on the choice of the adopted basis $\B$. 
\end{lem}

\proof
Consider another ordered basis $\B^\prime$.  Denote by $B$ and $B^\prime$ the matrices with columns elements of $\B$ and $\B^\prime$ respectively.   Then we have a change of basis matrix $H$ such that $H B=B^\prime$.  %
the interpolating matrix in the  new coordinate system is $H A H^{-1}$ and 
\begin{equation} 
\det (HAH^{-1} )=     \det(H)   \det(A)    \det(H)^{-1}   =   \det (A).
\end{equation} 
Thus the equation of the curve $\Y_m$ doesn't depend on the choice of the ordered basis $\B$.
\qed

Notice that $\det A$ is invariant (up to a sign change) under the permutations of points $P_1, \dots , P_m$. Hence, $\Y_m$ is defined over $k[s_i, \ldots , s_m]$, where $s_1, \dots , s_m$ are the symmetric functions on $P_1, \ldots , P_m$.   
Thus, the equation of $\Y_m$ is invariant under the permutations $(x_i, y_i) \to (x_j, y_j)$.  
We denote the invariants of  permuting the $x$-coordinates of  points $P_1, \dots , P_m$ by $s_1, \ldots , s_m$. In other words    
%
\[
s_1  = \sum_{i=1}^m x_i, \; s_2  =\sum_{i \neq j} x_i x_j, \; \ldots , \;  s_m  = x_1 \cdots x_m
\]
Let denote by $f_1, \ldots , f_4 $  and   $g_1, \ldots , g_4$ be defined as follows
\[
\begin{aligned}
& f_1 (x) = \prod_{i=1}^m (x-x_i),     			& 	&  g_1 (y) = \prod_{i=1}^m (y-y_i), \\
& f_2 (x) = \prod_{i=m+1}^{m+g} (x-x_i),     	& 	&  g_2 (y) = \prod_{i=m+1}^{m+g} (y-y_i), \\
& f_3 (x) = \prod_{i=m+g+1}^{m+2g} (x-x_i),     & 	&  g_3 (y) = \prod_{i=m+g+1}^{m+2g} (y-y_i), \\
& f_4 (x) = \prod_{i=3g+1}^{4g} (x-x_i),     		&	&  g_4 (y) = \prod_{i=3g+1}^{4g} (y-y_i), \\
\end{aligned}
\]
\begin{rem}\label{rem-2}
If $\X$ is defined over $k$ and $f_i (x) \in L[x]$, for some field $L \subset k$ and $i=1, \ldots , 4$, then the corresponding $g_i(y) \in L[y]$.  This is obvious, because $g_i (y)= \mbox{Res }(f(x, y), f_i (x), x)$. 
\end{rem}

Then we have the following.
\begin{prop}\label{thm-2}
Let $m\leq 2g$ and $P_1, \ldots , P_m \in \X$.  Then $\Y_m$ defined in \cref{eq-Y} is an algebraic curve with equation  $ \Y_m : \; g(x, y)=0$, 
where $g \in k (s_1, \ldots , s_m )  [x, y]$ and  $\deg g(x, y) = \deg \varphi_{m+g} (x, y)$. 
\end{prop}

\proof    Consider the determinant $\det A$. The coefficient of  the   $\varphi_i$ is $(-1)^{1+j}B_{1j}$, where $B_{1j}$ is the minor obtained by removing the $1$-st row and the $j$-th column. Recall that the poles of $\varphi_1, \dots , \varphi_m$ have at most order $g$.   
We  view the $\det A$ as a polynomial in $x$ and $y$ of degree $m+g$, since by clearing out denominators we can only have degree $g$ monomials.

Coefficients of $\Y$ are given as ratios of the minors $B_{1j}$, which are invariants under permutations of the points $P_1, \dots , P_m$.   This completes the proof. 
\qed

\subsection{Minimal field of definition for interpolating curves}
Let $D_1 = \sum_{i=1}^g P_i$ and $D_2= \sum_{i=g+1}^{2g} $ as above.  Denote by $P_i (x_i, y_i)$ for $i=1, \ldots , 2g$ and    by 
\[ 
f_1(x) = \prod_{i=1}^g (x-x_i), \qquad f_2 (x) = \prod_{i=g+1}^{2g} (x-x_i), \quad f_3(x)=  \prod_{i=2g+1}^{3g} (x-x_i)
\]
Let $L$ be a subfield of $k$ and assume that $f_1 (x),  f_2 (x) \in L[x]$.  From \cref{rem-2} then the polynomials determining the $y$-coordinates are also defined over $L$. 

\begin{lem}\label{lem-2}
Then  the polynomials  $f_3(x)$ and $f_4(x)$ are also in $L[x]$. 
\end{lem}

\proof   Notice that from part iii) of \cref{main-thm} the projective equation of $\X^{"}$ is given as
 \[ f_3 (x, y, z) = f_1 (x, y) \cdot A(x, y, z) + f_2 (x, y, z) \cdot B(x, y, z), \]
 for some $A, B \in L[x, y, z]$.  Hence, $f_3 (x, y, z) \in L[x, y, z]$.
 
 The proof for $f_4(x)$ goes similarly by applying \cref{main-thm} to  the curves $\X$ and $\X^{"}$.
\qed

Thus,  $f_4(x)$ is a degree $g$ polynomial defined over the ground field $k$.  
Let us produce an explicit formula for $f_4(x)$ first using the fact that $b_{3g}\neq 0$ divide $F(x)$ by the leading coefficient to obtain the polynomials 
\begin{equation}\label{eq-F1}
F_1(x) = f_1(x) \, f_2(x) \, f_3(x) =x^{3g}+\rho_1x^{3g-1}+\cdots +\rho_x+\rho_{3g}.
\end{equation}
As we showed above the coefficients of $F(x)$ are symmetric rational functions in the coordinates 
$(x_i,y_i)$. We show how to explicitly produce the formula for $f_4 (x).$ This gives an explicit addition law, as we express the coefficients of $f_4(x)$ explicitly as symmetric functions on $x_i,y_i.$ Before formulating our result we will recall some definitions and relations for symmetric functions:

For $x_1,...x_n$ we have  we define \textbf{ homogenous symmetric polynomials} 
\[
h_k(x_1,x_2,...x_n)=\sum_{1\leq i_1\leq i_2...\leq i_k\leq n}x_{i_1}...x_{i_k},  \; \text{for  } k>0,  \; h_0=1
\]
and \textbf{elementary symmetric polynomials} as
\[
e_i(x_1,...x_n)=\sum_{1<i_1<li_2...<i_l<n}\prod_{j=1}^ix_{i_1}x_{i_2}...x_{i_l}, \; \textbf{ and }  e_0=1.
\]
The main  identity connecting them is:
\begin{equation}
\sum_{i=1}^m (-1)^ie_i(x_1,...x_n)h_{m-i}(x_1,...x_n)=0
\end{equation}
For a proof see \cite[section 6.1]{fulton-1}. 
%
\begin{lem} 
The $x$-coordinates of the $\X\cap \X^\prime$  are roots of 
$f_3 (x)=\sum_{i=0}^g d_ix^{g-i}$, 
where $\rho_i$ are as in \cref{eq-F1} and 
\begin{equation}
d_k=\sum_{i=0}^k{\rho_i   h_{k-i}\left(x_1, \cdots, x_{2g}\right)}
\end{equation} 
\end{lem}  

\begin{proof}
We need to prove the identity in \cref{eq-F1} for $f_3(x)$ as claimed in the Lemma.   We have
\[
\prod_{i=1}^{2g}(x-x_i)=\sum_{i=0}^{2g}(-1)^ie_i(x_1,...x_{2g})x^{2g-i}
\]
and $e_i$ is the $i$-th symmetric polynomial. 
%
Recall that if we multiply two polynomials
$p_1(x)=\sum_{i=0}^na_ix^{n-i}$ and   $p_2(x)=\sum_{j=0}^m  b_j  x^{m-j}$, we  get  $p_1(x)p_2(x)=\sum_{k=0}^{n+m}c_k  x^{n+m-k}$, where 
\[
c_l=\sum_{p=0}^l a_p b_{l-p} (x_1\cdots x_{2g})
\]
Hence we need to show that: 
\[
\rho_l=\sum_{p=0}^l(-1)^{l-p}d_p e_{l-p}=\sum_{p=0}^l\sum_{i=0}^p(-1)^{l-p}\rho_ih_{p-i}e_{l-p}=\sum_{i=0}^p\sum_{p=0}^l(-1)^{l-p}\rho_ih_{p-i}e_{l-p}
\]
From the last expression we conclude that the coefficient of  $\rho_i$, for $i\neq l$,  is given by 
$\sum_{p=0}^l(-1)^{l-p}h_{p-i}e_{l-p}$. 
Using the identity between the homogenous and elementary symmetric polynomials conclude that the last sum is $0$,
unless $i=l$.
%
\end{proof}


\section{Addition on Jacobian varieties}
Let $\X$ be a  smooth, irreducible, algebraic  curve of  genus $g\geq 2$, defined over an algebraically closed  field $k$.   Let $S_d$ denote the symmetric group of permutations.  Then $S_d$ acts on $\X^d$ as follows:
\begin{equation}
\begin{split}
S_d \times \X^d & \rightarrow \X^d \\
\left( \sigma , \left(  P_1, \dots , P_d \right) \right) & \rightarrow \left(\dots ,  P_i^{\sigma}, \dots \right) \\
\end{split}
\end{equation}
We denote the orbit space of this action by $\Sym^d (\X)$.

Denote by   $\Div^d (\X)$  the set of degree $v$ divisors in $\Div (\X)$ and by $\Div^{+, d} (\X)$ the set of positive ones in $\Div^d (\X)$.
Then,   $\Div^{+ \, d} (\X) \cong \Sym^d (\X)$. 
Let 
\[ j : \X^d \hookrightarrow \P^{(n+1)d -1 } \]
 be the Segre embedding. Let $R:= k [\X^d] $ be the homogenous coordinate ring   of $\X^d$.  Then $S_d$ acts on $R$ by permuting the coordinates.  This action preserves the grading.  
Then $j$ is equivariant under the above action.     Hence,    the ring of invariants $R^{S_d}$  is finitely generated by homogenous polynomials $\varphi_0, \dots , \varphi_N$ of degree $M$. Thus,
\[
k [ \varphi_0, \dots , \varphi_N] \subset \{ f \in R^{S_d } \, \text{ such that }  \, M | \deg f \} \subset R^{S_d } 
\]
Hence, every element in $k [ \varphi_0, \dots , \varphi_N]$ we can express it as a vector in $\P^N$ via the basis $\{\varphi_0, \dots , \varphi_N\}$.
Then we have an embedding 
\[ \Sym^d (\X) \hookrightarrow \P^N \]
with the corresponding  following diagram
\[
\xymatrix{  
\X^d   \ar[d]     \ar@{^{(}->}[r]^j            &       \P^{(n+1)d -1}     \ar[d]           \\
\Sym^d (\X)   \ar@{^{(}->}[r]      &        \P^N   \\
   }
\]
Thus, any divisor $D \in \Div^{+ \, d} (\X)$ we identify with its correspondent point in $\Sym^d (\X)$ and then express it in coordinates in $\P^N$.  
The variety $\Sym^d (\X)$ is smooth because $\Sym^d (\X) \setminus\{\Delta = 0 \}$ is biholomorphically to an open set in $k^d$.   Then we have:
\begin{prop} Let $\X$ be a genus $g\geq 2$ curve. The map 
\[ 
\begin{split}
\phi : & \Sym^g (\X) \longrightarrow \Jac \X \\
& \sum P_i \longrightarrow \sum P_i - g \infty \\
\end{split}
\]
is surjective.  In other words, for every divisor $D$ of degree zero, there exist $P_1, \dots , P_g$ such that 
$D$ is linearly equivalent to $\sum_{i=1}^g P_i - g\infty$. 
\end{prop}

The  result was suggested by Jacobi. For a proof  see  \cite{tata-2}.  Since every divisor can be expressed as a reduced divisor, it is enough to define the addition in $\Jac \X$ among  divisors of the form   $\sum_{i=1}^g P_i - g\infty$.  Next we will make this precise by using the intersection theory of the previous section.

Fix a point $P\in \X (k)$ and let 
\[ \iota_P : \X \hookrightarrow \Jac_k \X , \]
be the corresponding polarization.    A divisor $D \in \Pic_k (\X)$ is called a \textbf{reduced divisor with respect to $\iota_P$}  if it is written in the form 
\[ D = \sum_{i=1}^g P_i - g P,\]
for $P_1, \dots , P_g \in \X$.    Notice that we are not requiring that $P_1, \dots , P_g$ are all distinct.  
Consider now the following problem:   \\

\noindent \textbf{Problem:}  \textit{Given two reduced divisors   
\[ D_1=\sum_{i=1}^g P_i - g P  \quad \text{ and } \quad D_2 = \sum_{i=g+1}^{2g} P_i    -    g  P,\]
determine a reduced divisor $D=\sum_{i=2g+1}^{3g} P_i-g P$ such that $D=D_1+D_2$.  \\
}

We follow the following strategy.  From  $D_1$ we get $P_1, \ldots , P_g$ points on $\X$ and from $D_2$ the other $P_{g+1}, \ldots , P_{2g} \in \X$. 
 From theorem \cref{main-thm} exists $\X^\prime$ such that $\X$ intersects $\X^\prime$ transversally in $3g$ points.   That means that we have $g$ new points of intersection, say 
 $P_{2g+1}, \dots , P_{3g}$.  Define 
\[ D^\prime := \sum_{i=2g+1}^{3g} P_i -   g P.\]
 Then obviously $D^\prime := - (D_1+D_2)$.  

From these points we get the curve $\X^{"}$ as in the proof of \cref{main-thm},   which intersects $\X^\prime$ transversally in $g$ new points $Q_1, \ldots , Q_g$ and intersects  $\X$ as
\[
\X \bigcdot \X^{"} = \left( \sum_{g+1}^{2g} P_1 \right) + \sum_{j-1}^g Q_j. 
\]
Then $D := - D^\prime = \sum_i^g Q_i - g P$. 
%



With the above discussion, the addition map in \cref{add-map}  induces an map on Jacobians, namely
\[
\begin{split}
\varphi :  \Jac \X \times \Jac \X & \to \Jac \X \\
\left(  \sum_{i=1}^g P_i - g \infty,       \sum_{i=g+1}^{2g} P_i - g \infty      \right)  & \to   \sum_{i=1}^{g} Q_i - g \infty
\end{split}
\]

Hence, we have the following:
\begin{prop}
For any two reduced divisors $D_1$ and $D_2$ in $\Jac (\X)$, the addition is given by 
\[ D_1 + D_2 = \varphi (D_1, D_2).\]
Moreover,  if $D_1$ and $D_2$ are defined over any field  $L\subset k$, then $D_1 + D_2 $ is defined over $L$.  
\end{prop}

\proof The first part comes from the fact that   the intersection divisor $\X\cap \X^\prime$ is $-(D_1+D_2)$.  Moreover, the intersection divisor 
\[
\X^\prime \cap \X^{"}= \varphi (D_1, D_2) = D_1+D_2.
\]
The last part  is an immediate consequence of the above \cref {lem-2}
\qed


\subsection{Inverting divisors}
In previous sections, for any two reduced divisors $D_1$ and $D_2$ we found a reduced divisor $D$ such that $D_1+D_2 + D =0$.  To get precise addition formulas for $D_1 + D_2$ we have to invert $D$ (i.e., find $-D$).  As suggested in Section~2, we apply \cref{thm-1} for $s=0$.

\begin{prop}
Let $D=\sum_{i=1}^g P_i - g \infty$ be a reduced divisor, where $P_1, \ldots , P_g \in \X$.  Let $\varphi_1, \dots , \varphi_g$ be the first $g+1$ monomials of the basis $\B$ of $k(\X)/k$ ordered according to their order at $\infty$ and 
\[ 
\mathcal Z : \det A_{(\varphi_1, \dots , \varphi_{g})} (P_1, \dots , P_{g}) =0.
\]
  Then $\mathcal Z$ intersect $\X$ precisely on $g$ points  $Q_1, \dots , Q_g$ in addition to $P_1, \dots , P_g$.  Moreover, 
\[ -D = \sum_{i=1}^g Q_i - g\infty.\]

\end{prop}

\proof  The proof is a direct consequence of \cref{thm-1} for $s=0$.  The basis is determined the same way but now we use only the first $g+1$ functions in the definition of $\Y$. 
\qed

\section{Superelliptic curves}\label{superelliptic}
Let $\X$ be a genus $g\geq 2$  defined over $k$ such that there exists an order $n>1$ automorphism $\sigma \in \Aut (\X)$ with the following properties:   i) $H:=\<\s\>$ is normal in $\Aut (\X)$, and ii) $\X/\<\s\>$ has genus zero.  Such curves are called superelliptic curves and their Jacobians, superelliptic Jacobians.  They have affine equation 
\begin{equation}\label{super}
\X  : \; y^n = f(x) = \prod_{i=1}^d (x-\a_i)
\end{equation}
We denote by $\sigma$ the superelliptic automorphism of $\X$.  So $\sigma : \X \to \X$ such that 
$ \sigma (x, y) \to (x, \xi_n y)$, 
where $\xi_n$ is a primitive $n$-th root of unity.  Notice that $\sigma$ fixes 0 and  the point at infinity in $\P_y^1$.    The natural projection $\pi : \X \to \P^1_x=\X/\<\sigma\>$ has  $\deg \pi =n$ and 
$ \pi (x, y) =  x$.
This cover is branched at exactly at the roots $\a_1, \dots , \a_d$ of $f(x)$. 

If the discriminant $\Delta (f, x) \neq 0$ and $d>n$   then from the Riemann-Hurwitz formula we have
\[ g = \frac 1 2 \left(  n(d-1) - d - \gcd (n, d)    \frac {} {}  \right)     + 1 \]  
There is a lot of confusion in the literature over the term \textit{superelliptic} or \textit{cyclic} curves.  To us a \textit{superelliptic curve} it is a curve which satisfies \cref{super} with discriminant $\Delta (f, x) \neq 0$.  

If $\gcd (n, d)=1$ then $\deg f$ is either  $\frac {2g} {n-1}+ 2$ or $ \frac {2g} {n-1}+ 1$, depending on whether or not the place at infinity is a branch point of the superelliptic projection map. We will always assume that infinity is a branch point.    We denote the set of roots of $f(x)$ by $\B=\{ \a_1, \ldots , \a_d\}$. 

\begin{prop}\label{prop-2}
Let $\X$ be a superelliptic curve with equation \cref{super}, s.t.  $\Delta (f) \neq 0$,  $\deg f =d> n$, and let $d=sn -e$, for $0< e< n$.
Then a  basis for the space of holomorphic differentials is 
\[ 
\left\{ x^{i}\frac{dx}{y^{j}} \;  | \; 1\leq j\leq n,  \; 1\leq i\leq b_j  \right\},
\]
where $b_j =sj - 1 - \left\lfloor \frac e n j \right\rfloor$.
\end{prop}

See \cite[Prop.~2]{towse} for the proof.  This basis can easily be converted to a monomial basis 
\[ 
\left\{ x^{i} y^{n-j}   \;  | \;   1    \leq  j     \leq n,  \; 1\leq i\leq b_j  \right\},
\]
by clearing the denominators.    This gives us $g$ meromorphic functions. To get the rest of the $2g$ meromorphic functions  we take  functions whose order at $\infty$ is between $2g$ and $3g$. We have the following. 
\begin{prop} \label{prop-3}
For every order $j$ at $\infty$ such that  $2g\leq j\leq 3g$ we have a monomial $x^m y^{m_j}$ such that the order of this monomial at $\infty$ is exactly $j$.
\end{prop} 

\begin{proof} 
First note that it is enough to show the proposition for $2g\leq j \leq 2g+n-1$. Indeed, assume that $2g+n\leq j\leq 3g$. 
We have a unique integer  $2g\leq r \leq 2g+(n-1)$ such that $j=r+vn$, for some $v\in \Z$.   
If the polynomial $x^{m}y^{m_j}$ corresponds to $r$ then $x^{m+v}y^{m_j}$ will correspond to $j$. 

Hence, we assume $2g\leq j\leq  2g+n-2$.  Let $r=j-2g$ and consider 
\[
 (d-1)(n-1)-d(r+1)+d(r+1)+r = d(n-2-r)+d-(n-1)+r(d+1)
\]
Now consider the monomial: $y^{n-2-r}x^{(s-1)+rs}$ calculating the order at $\infty$ gives us the last expression. 
All is left is to find a monomial for 
\[ 2g+(n-1)=(d-1)(n-1)+(n-1)=(d-1+1)(n-1)=d(n-1)\] 
for this the monomial   $y^{n-1}$ would do the trick. 
\end{proof}


Let $L (k\infty)$ denote the space of meromorphic functions on $\X$ which are holomorphic on $\X\setminus \{\infty\}$ and have poles of order at most $k$ at $\infty$.  From the Riemann-Roch we have
\[ \dim (  L (N+g-1) \infty ) = N, \;  \text{ for } \; N \geq g.\]
Consider the space 
\[   L (\star \infty)  := \cup_{k=1}^\infty  L  (k \infty),    \]
of meromorphic functions on $\X$ which are holomorphic on $\X\setminus \{ \infty \}$.  This is the space of polynomials on $x$ and $y$.   Then we have the following.

\begin{lem} A basis of $L (k\infty)$ over $k$ is given by 
\[ \B:=\left\{    x^i y^j, \;    0 \leq i\leq d, \;  0 \leq j \leq n-1.   \right\}\]
\end{lem}

\proof   The proof follows from the remarks above. 
\qed

We can put these monomials in a matrix $B=[b_{i, j}]$ such that  $b_{i, j} = x^i y^j$.
So the matrix will have $n$ rows and at most $d+1$ columns and in the  $j$-th row  it will have monomials $y^{j-1} x^i$, for $i=0, 1, \ldots d$.  
For a meromorphic function  $ f = x^i y^j$,   the $\ord_\infty f$ is  
\[   \ord_\infty x^i y^j = ni + dj.\]   
In particular,
\[ \ord_\infty x^i = n\cdot i \quad \text{ and } \quad \ord_\infty y^j = d \cdot j.\]
We order the basis of $L (\star \infty)$ according to the order at $\infty$.  Let $\{\varphi_i\}$ be the monomial basis of $L (\star \infty)$ ordered as
\[ 0=\ord_\infty \varphi_1 <  \ord_\infty \varphi_2 < \ord_\infty \varphi_3 < \dots .\]
Notice that $\ord_\infty 1=0$, $\ord_\infty x = n$, $\ord_\infty y = d$.  The first monomials will be 
\[ 
1, x, \dots , x^r, y, \dots 
\]
for $r=\left\lfloor \frac d n \right\rfloor $.   Hence, if we fill the matrix $B$ only with the first $2g+1$ monomials and assign zeroes to all the other entries then we call it the \textbf{corresponding matrix} to the  curve $\X$ and denote it by $B_{\X}$   or in case of superelliptic curves $B_{n, d}$.
For a given curve $\X$ we want to determine it corresponding matrix $B_{\X}$ .  
 Next we see an example. 
\begin{exa}
Consider $n=4$ and $d=13$.  Then we have a curve $\X$ of genus $g=18$.  The possible orders of monomials $x^iy^j$ at $\infty$ are
\[ 
\begin{split}
& 0, 4, 8, 12, 13, 16, 17, 20, 21, 24, 25, 26, 28, 29, 30, 32, 33, 34, 36, 37, 38, 39, 40, 41, 42,\\
&  43, 44,  45, 46, 47, 48, 49, 50, 51, 52, 53, 54;   \\
&
55, 57, 58, 59, 61, 62, 63, 65, 66, 67, 70, 71, 74, 75, 78, 79, 83, 87, 91.
\end{split}
 \]
The first $2g+1$   monomials are:
\[
\begin{split}
& 1, x, x^2, x^3, y, x^4, xy, x^5, x^2y, x^6, x^3y, y^2, x^7, x^4y, xy^2, x^8, x^5y, x^2y^2, x^9, x^6 y, x^3 y^2, \\
& y^3, x^{10}, x^7 y, x^4 y^2, x y^3, x^{11}, x^8 y,  x^5 y^2, x^2 y^3, x^{12}, x^9 y, x^6 y^2, x^3 y^3, x^{12}, x^9 y, x^6 y^2, \\
&  x^2 y^3, x^{13}, x^{10} y, x^7 y^2. 
\end{split}
\]
However, if we rearrange the monomials to their monomial ordering we have 
\[ 1, x, x^2, \ldots , x^{13}, y, yx, yx^2, \ldots, yx^{10}, y^2, y^2x, \ldots, y^2x^7, y^3, y^3x, y^3 x^2, y^3 x^3.\]
The matrix $B$ in this case is
\[
B_{4, 13}=\begin{bmatrix*}
1	& x		& x^2 	& x^3 	& \ldots & x^7	& \ldots & x^{10} 	& \ldots 	& x^{11} & \ldots 	&  x^{13} 	\\
y   & xy 	& x^2y 	& x^3 y & \ldots & x^7y & \ldots & x^{10} y &  0 		& 0 	 & 0 		&  0	    \\
y^2   & xy^2 	& x^2y^2 	& x^3 y^2 & \ldots & x^7y^2 & 0 & 0 &  0 		& 0 	 & 0 		&  0	   \\
y^3   & xy^3 	& x^2y^3 	& x^3 y^3 & 0 & 0  & \ldots & \ldots  &  \ldots		&  \ldots 	 &  \ldots 		&  0	   
\end{bmatrix*}
\]
\qed
\end{exa}

We try to generalize for the case $B_{n, d}$.   Assuming $\deg x =n$ and $\deg y = d$ we explicitly give the first $2g+1$ monomials. 

\begin{thm} \label{pattern}
Let $\X$ be a superelliptic curve with affine equation $y^n = f(x)$, where $\deg f = d$ and $(n, d)=1$.  Then $B_{n, d}$ is an $n \times (d+1)$ matrix and  the non-zero entries in the $j$-th row, for $j=0, \dots , n-1$,      are given by  monomials are given by $x^i y^j$  for  $0 \leq i \leq  \left\lfloor \frac{3g-jd}{n} \right\rfloor$.
\end{thm}

\proof 
We divide the proof into two parts. First we show that for $m\geq (d-1)(n-1)$ there is exactly one pair of non-negative integers $\rho,\sigma$ such that $\sigma\leq (d-1)$ and $m=\rho n+\sigma d.$ Since $\gcd(n,d)=1$ all of integers of the form $m-jd, j=0,1...n-1$ are mutually distinct $\mod n.$ Hence there is a unique value of $j=\sigma$ 
and a non-negative $\rho$ such that $m=\rho n+\sigma d$.   Because we have that
\[ \rho n \geq (n-1)(d-1)-(n-1)d=1-n,\]
we can  conclude that $\rho$ is positive. Hence it follows that for any $m\geq (n-1)(d-1)$ we have a unique polynomial of the form $x^ly^c$ and $c\leq (d-1)$.

Now if $m\leq (d-1)(n-1)$ it follows from Sylvester theorems on semi-groups; see \cite{sylvestre}. 
\qed

\begin{rem} 
In the example above suppose we want to determine all the monomials of the form $x^jy^2$ as we know that degree of $y$ is $13$ we have $54-26=28$ and $\frac{28}{4}=7$ which produces the result. 
\end{rem}

Notice that \cref{pattern} completely determines the first $2g+1$ monomials of a superelliptic curve, ordered according to their order at $\infty$.  It is probably possible to determine the matrix $\B_\X$ for any curve $\X$.

\begin{cor}
The degree of the curve $\Y$ is given by 
\[
\deg \Y = \max   \left\{     \frac {3g-j (d-n)} n \; : \; 0 \leq j \leq  n-1,     0 \leq i \leq \left\lfloor  \frac {3g-jd} n \right\rfloor  \right\}
\]
\end{cor}

The following is known from \cite{cantor-1},  \cite{Leitenberger}. 

\begin{cor}
$\Y$ has genus zero if and only if $\X$ is hyperelliptic.   In this case $y$ is given as a rational function in $x$. 
\end{cor}

We summarize the general case  in the following:

\begin{thm}
Let $\X$ be a genus $g\geq 2$ superelliptic curve , $D_1, D_2 \in Jac_k (\X)$ written in reduced form as 
\[ D_1=\sum_{i=1}^g P_i - g \infty \quad \text{ and } \quad D_2=\sum_{i=g+1}^{2g} P_i - g \infty, \]
where $P_1, \ldots , P_{2g} \in \X$  are all distinct.  Let $\varphi_1, \dots , \varphi_{2g+1}$ be the  first $2g+1$ nonzero entries of $B_{n, d}$ as in \cref{pattern}
and 
\[ \Y : \det A_{(\varphi_1, \dots , \varphi_{2g+1})} (P_1, \dots , P_{2g}) =0.\]
Then $\Y$ intersect $\X$ precisely on $g$ points  $P_{2g+1}, \dots , P_{3g}$ in addition to $P_1, \dots , P_{2g}$.  Moreover,  $D_1 + D_2 \in \Jac_k (\X)$ is given by 
\[ D_1+D_2= -     \left(   \sum_{i=2g+1}^{3g}    P_i  - g \infty \right) .\]
\end{thm}

\proof
We already showed above that over the algebraic closure $D_1+D_2$ is as claimed.    Since $D_1$ and $D_2$ are defined over $k$, then the polynomial 
\[ g_1 (x) = \prod_{i=1}^{2g} (x-x_i), \]
is defined over $k$.  To find the intersection of $\X$ with $\Y$ we take the resultant  with respect to $y$ of the two corresponding equations and get a polynomial   $F(x)$ 
defined over $k$ and with degree $\deg F= 3g$.  The $x$-coordinates of $P_1, \ldots , P_{3g}$ are roots of this polynomial.  Hence, the $x$- coordinates of   $P_{2g+1}, \ldots , P_{3g}$ are roots of
\[  g_2 (x) = \frac {F(x)} {g_1 (x) },\]
 which is also defined over $k$.  Hence $D_1+D_2$ is also defined over $k$. 
\qed

\subsection{Repeating points}

So far our approach will work for any generic collection of points     $P_1, \dots , P_{2g} \in \X$.  Thus, our  points $P_1, \dots , P_{2g}$ are all distinct.  However, if two of them are the same then the matrix $A (P_1, \ldots , P_{2g})$ will have two identical rows and therefore $\det A =0$. To remedy the situation assume $P_i=\left(x_i,y_i\right),P_{i+1}=\left(x_{i+1},y_{i+1}\right)$ are coming together. That is we are looking on $P_i^2$ instead of $P_iO_{i+1}.$ Define $t_i=x_{i+1}-x_i$ then if $x_{i+1}^ly_{i+1}^m$ is a monomial we can decompose it as: 
\[
x_{i+1}^ly_{i+1}^m=x_i^ly^m+t\frac{\partial{x_{i}^ly_{i}^m}}{\partial x}+O(t^2).
\]
 We can calculate the derivative using the product rule and implicit function theorem as:  
\begin{equation}\label{derivative}
\frac{\partial{x_{i}^ly_{i}^m}}{\partial x}=lx_{i}^{l-1}y_i^m+mx_i^ly_i^{m-1}\frac{\frac{\partial F(x,y)}{\partial x}}{\frac{\partial F(x,y)}{\partial y}}
\end{equation}
Taking $t\to 0$ as points are coming together we replace the monomial $x_{i+1}^ly_{i+1}^m$ with the expression: 
\begin{equation}\label{derivative}
lx_{i}^{l-1}y_i^m+mx_i^ly_i^{m-1}\frac{\frac{\partial F(x,y)}{\partial x}}{\frac{\partial F(x,y)}{\partial y}}
\end{equation}
For each monomial $\varphi_i$ in the ordered basis $\B$ we denote the corresponding monomial as in \cref{derivative} by $g_i$.  
For occurrences of higher order we successively replace the points with derivatives of higher order similar to the first order.

\subsection{Hyperelliptic curves} 
As an application to our method,  let us now consider the simplest case of superelliptic curves, namely $n=2$.    From above  we have that the  list the non-gaps for hyperelliptic curves are: 
\[
0, 2, 4, 6, \dots , 2g,   2g+2, \dots 
\] 
The function field $k(\X)$ is generated by 
\[  
\B = \{ 1,x,x^2,x^3, \dots , x^g,  y,   yx, yx^2,  yx^3, \dots ,       y x^g \}. 
\]
We take   these monomials according to increasing order at $\infty$, which  is given by 
\[ \ord_\infty x^i y^j = 2i+(2g+1) j .\]
Then, we can reorder $\B$ ordering according to $\ord_\infty$ and have the following:

\begin{lem} 
Let $\X$ be a genus $g\geq 2$ hyperelliptic curve and $s:= \left\lfloor  \frac {g-1} 2 \right\rfloor $.
The first $2g+1$ monomials of the basis $\B$,  ordered according to their order at $\infty$ are 
\[  
1,x,x^2,x^3, \dots , x^g,  y,  x^{g+1},  yx, x^{g+2},   yx^2,  x^{g+3}, yx^3, \dots ,      x^{g+s},  y x^s, 
\]
if $g$ is odd and 
\[  
1,x,x^2,x^3, \dots , x^g,  y,  x^{g+1},  yx, x^{g+2},   yx^2,  x^{g+3}, yx^3, \dots ,      x^{g+s},  y x^s, x^{g+s+1}
\]
if $g$ is even. 
\end{lem}

\proof  The proof is a direct consequence of \cref{pattern} by taking $n=2$ and $d=2g+1$. 
\qed

Let $P_i :=(x_i, y_i)$, $i=1, \dots , 2g$ and consider  the matrix $A$ as defined in \cref{matrix-A}.  As before $\Y : \det A (P_1, \ldots , P_{2g}) =0$.    Notice that the equation of $\Y$ is linear in $y$.  Hence, $y$ can be expressed as a rational function %
\[
y=\frac {h(x)} {g(x)},
\]
 where $\deg h = g+s$ when $g$ is odd and $\deg h = g+s+1$ when $g$ is even.  The degree of the denominator is  $\deg g = s$.   

\begin{rem} 
The addition of divisors in hyperelliptic  Jacobians is done via   Cantor's algorithm described in \cite{cantor-1}.   A geometric interpretation of that addition   is given in \cite{Leitenberger}.
The results here match exactly  those in   \cite{cantor-1} and \cite{Leitenberger}, where the interpolating curve becomes and interpolating rational function.   
\end{rem}

\subsubsection{Genus 2 curves}
We apply the theory developed previously to an explicit addition law for hyper-elliptic curves of genus $2$. For a   a genus $2$ curve with $P=\infty$ the equation is: 
\[ y^2= x^{5} + a_1 x^{4} + \cdots a_{4} x + a_5. \]
The usual basis $\B$ adapted to $P=\infty$ can be taken as 
\[ \B = \{ 1,  x, x^2, x^3, y \}, \]
as noted above.  However we use the fact that our curve $\Y$ doesn't depend on the basis and se an alternative basis for the $x$ part of the $\B$. Consider first a sum of two divisors each one of  degree $4$. The non-reduced form will be a divisor of degree $4$ respectively, $D_{4}=\sum_{i=1}^{4}P_i$,    $P_i\neq P_j$.
 Our goal is to find an explicit equation depending on the $x_1, \ldots, x_4$ of degree $2$ such that the roots of this equation will be the $x$ coordinate of the divisor of degree $2$ equivalent to the divisor $D_{4}$.

\begin{definition}
Let $x_1, \cdots x_n$ be an arbitrary complex numbers. and let $q(x)=\prod_{i=1}^n(x-x_i).$ For each $x_i$ define the Lagrange polynomial $l_i (x)$ as: 
\begin{equation}
l_i(x)=\frac{q(x)}{q'(x_i)(x-x_i)}=\frac{\prod_{l=1,l \neq i}^n(x-x_l)}{\prod_{l=1,l\neq i}^n,(x_i-x_l)}
\end{equation}
\end{definition}
We have the following lemma.
\begin{lem}
Let $P_{n-1}[x]$ be the vector space of polynomials $p(x)$ such that $\deg p   \leq (n-1)$.  Then   $l_i(x)|1\leq i\leq n$ is a basis  $P_n [x]$
\end{lem}

\begin{proof}
For each of the points $x_1, \ldots, x_{n}$ consider the functional $\phi_i (x)=p (x_i)$, for  $p \in P_{n-1}[x]$.      Then, $\phi_i   (l_j(x))=\delta_{i j}$.  Hence $l_i(x)   \leq (n-1)$ are dual to the functionals $\phi_i (x)$,   $i=1...{n-1}$.  Hence they are a basis for $P_{n-1}[x]$.
\end{proof}

Apply the last lemma to our situation and replace the basis $\B$ with a basis 
\[ \B_1 = \{ l_1(x), \ldots ,  l_4(x), y \},.\]
In this basis we have the lemma:
\begin{lem}\label{curve2eq}
The equation for the interpolating curve $\Y$ is: 
\begin{equation}\label{hypeA}
y-\sum_{i=1}^4y_il_i(x)=0
\end{equation}
\end{lem}

\begin{proof}
First note that in the the basis $B_1$ the matrix $A$ has the following form: 
\begin{enumerate}
\item $A_{i+1i}=1$,   for  all  $2\leq i\leq 4$ 
\item $A_{1i}=l_i(x)$,    for   all $1\leq i\leq 3$ and $A_{1,4}=y$
\item $A_{i,2g}=y_{i-1}$,       for all $i\geq 1$
\end{enumerate} 
Expanding this matrix with respect to the first row we see that the $y$ coefficient is  $1.$ The other minor's equal $y_i$ as desired. 
\end{proof}

To find the intersection points of $\Y$ with the genus $2$ curve $\X$,     substitute 
\[
y=\sqrt{\prod_{i=1}^{5}\left(x-\lambda_i\right)}
\]
in equation \cref{hypeA}.   Conclude that $x$ coordinates of the intersection points between  the curves $\Y$  and    $ \X$ satisfy the following equation: 
\begin{equation}
-\left(\sum_{l=1}^{4} y_jl_j(x)\right)^2=0
\end{equation}
Expanding the last equation we obtain: 
\begin{equation}
x^{5} + a_1 x^{4} + \cdots a_{4} x + a_5-\sum_{l=1}^{4}y_j^2l_j^2(x)-2\sum_{i,j=1,i\neq j}^{4}y_iy_jl_j(x)l_i(x)=0
\end{equation} 
The polynomials $2\sum_{i, j=1, i \neq j}^{4}y_i  y_j  l_j  (x)  l_i(x)$ are divisible by $\prod_{i=1}^{4}  \left(x-x_i\right)$.   To calculate the result of this division we define: 
\begin{definition}
For each $1\leq i,j\leq 4$ and $i\neq j$  let 
\begin{equation}
l_{i,j}=\prod_{l=1, l\neq i,j}^{4}\frac{(x-x_l)}{(x_l-x_i)(x_l-x_j)}\times \frac{1}{(x_i-x_j)^2}
\end{equation}
\end{definition}  
We have that: 
\begin{lem}The formulas for  $l_{i, j}$ are given by
\begin{equation}
l_{i,j}=\frac{l_{i}(x)l_j(x)}{\prod_{i=1}^4(x-x_i)}.
\end{equation}
\end{lem}

\begin{proof}
Straightforward computation. 
\end{proof}

Now let us look on the other part of the sum namely: 
\[
\prod_{i=1}^{5}\left(x-\lambda_i\right)-\sum_{l=1}^{4}y_j^2l_j^2(x)
\]
 as $x_j, j=1\cdots 4 $ is a root of this polynomial is divisible by $\prod_{j=1}^4\left(x-x_i\right).$ We will use the \cref{eq-F1} to obtain the coefficient of the quadratic polynomial. Let us write:  
\begin{equation}
l_j(x)=\sum_{i=0}^3\alpha_{ji}x^{3-i}
\end{equation}
where  $\alpha_{ji}$ are symmetric polynomials  in $x_k$,  $1\leq k\leq 4, k\neq j$.  Use the polynomial 
\[
q(x)=\prod_{i=1}^4(x-x_i)=x^4+b_1x^3+b_2x^2+b_3
\]
 to conclude that:
\begin{equation}
\alpha_{j i}=\frac{(-1)^i  e_i(x_k)}{q'(x_j)},   \quad 1\leq k\leq 4,      k\neq j 
\end{equation}
Hence we can express the coefficients of  $\sum_{j=1}^4  y_j^2   l_j^2$ as symmetric polynomials in $x_1, \ldots, x_4$.

Let   $u(x)=x^{5} + a_1 x^{4} + \cdots a_{4} x + a_5-\sum_{l=1}^{4}y_j^2l_j^2(x)$. Hence $u(x)=x^6+u_1x^5+\cdots u_6$.  Recall the definition of $q(x)$ 
\[
q(x)=\prod_{i=1}^4=x^4-q_1x^3+q_2x^2+_3x+q_4.
\]

\begin{lem}
$u(x)$ is divisible by $q(x)$ and if  $h(x)=\frac{u(x)}{q(x)}$ then 
\begin{enumerate}
\item $h_1=u_1-q_1$
\item $h_2=u_2-u_1q_1+q_1^2-q_2.$
\end{enumerate} 
\end{lem}

\begin{proof}
To show that $u(x)$ is divisible by $q(x)$  it is enough to show that $u(x_k)=0$ for all  $k$.    We have that    $l_j(x_k)=\delta_{ik}$. Hence,
\begin{equation}
\begin{split}
u(x_k)  &   =x_k^{5} + a_1 x_k^{4} + \cdots a_{4} x_k + a_5-\sum_{i=1}^{4}y_i^2l_i^2(x) \\
& =y_k^2-\sum_{i=1}^{4}y_i^2l_i^2(x_k)=y_k^2-\sum_{j=1}^4y_j^2\delta_{ik}=y_k^2-y_k^2=0
\end{split}
\end{equation}
To show the second part of the lemma write $u$ as
\[
u(x)=(x^2+h_1x+h_2)(x^4+q_1x^3+q_2x^2+q_3x+q_4).
\]
 Then we have that
\begin{enumerate}
\item $u_1=q_1+h_1$
\item $u_2=h_2+h_1q_1+q_2$
\end{enumerate}
where $q_1=u_1-h_1$ and the assertion for $h_2$ is obtained solving for it using the expression above. 
\end{proof}
The following is an immediate corollary of the preceding lemmas.
\begin{cor}
If $x_1,\cdots x_4$ are the $x$ coordinates of  $P_1P_2P_3P_4$ then the equation satisfied by its $x$ coordinates of its reduction, $P_5P_6-2\infty$ is: 
\begin{equation}
\sum_{i,j=1}^{4}y_iy_jl_{i,j}(x)+h(x)=0
\end{equation}
\end{cor}

\subsubsection{Genus 3 hyperelliptic}

Let us see now how things work out for genus 3 hyperelliptic curves. 
Let $\X$ be the genus 3 hyperelliptic curve with equation $y^2= f (x)$, where $\deg f =7$.   Then $n=2$  and $d=7$.  The matrix $B_{2, 7}$ is 
\begin{equation}\label{B-2-7}
B_{2,7}=
\begin{bmatrix}
1 & x& x^2 & x^3 & x^ 4\\
y & yx & 0 & 0 & 0
\end{bmatrix}
\end{equation}
So we have the first seven orders at $\infty$ as 
\[ 0, 2, 4, 6, 7, 8, 9 \]
and the corresponding monomials   $\{ 1, x, x^2, x^3, y, x^4, yx \}$.  hence our basis will be 
\[ \B= \{ 1, x, x^2, x^3, y, x^4, yx   \} \]
In this case, $\Y$ will be a curve with equation of the form
\[ c_0 + c_1 x + c_2 x^2 + c_3 x^3 + c_4 y + c_5 x^4 + c_6 yx =0.\]
Hence,  
\[ y = -    \frac {c_0 + c_1 x + c_2 x^2 + c_3 x^3 + c_5 x^4 } {c_4 + c_6 x}
\]
is a rational function $y= \frac {h(x)} {g(x)}$, where $\deg h =4$ and $\deg g =1$.   
%

Both the above cases are of special interest in hyperelliptic curve cryptography. 
For a survey in hyperelliptic curve cryptography and effectiveness of addition on hyperelliptic Jacobians  see \cite{frey-shaska} among many other sources. 

\subsection{Triagonal superelliptic curves}
Using the above approach it turns out that the case of triagonal superelliptic curves is a very simple case. Surprisingly it has not appeared in the literature before. 

Assume $n=3$.  Then $\X$ has equation $y^3=f(x)$ for $\deg f =d$.  We assume that $(n, d)=1$. Then, $\X$ has genus $g=d-1$. 

\begin{lem}\label{lem-n-3}
For triagonal curves with equation $y^3=f(x)$ such that $\deg f =d$, the first $2g+1$ monomials of our basis $\B$ are  
\[ 1, x, x^2 , \dots , x^{d-1}, \; y, yx, \dots , yx^s, \;  y^2, y^2x, \dots , y^2 x^q,\]
where $s$ and $q $ are as follows:

i) if $d \equiv 1 \mod 3$ then $q= \frac {d-1} 3$ and $s= 2 \frac {d-1} 3$.

ii) if $d \equiv 2 \mod 3$ then $q= \frac {d-2} 3$ and $s = \frac {d-5} 6$
\end{lem}

\proof  The proof is a direct consequence of \cref{pattern} by taking $n=3$. 
\qed

Notice that in both cases $s+q = d-1$ and $q = \left\lfloor \frac s 2 \right\rfloor$.    We define $\Y$ as before.  Then
$\Y$ is a hyperelliptic curve of genus $\frac {d-1} 2$ or $\frac {d-3} 2$.

\subsection{Picard curves}
A Picard curve has a degree three superelliptic projection $\pi : \X \to \P^1$.  This covering has five branch points, one of which we have specified at infinity. 
 The curve has equation
\begin{equation}\label{X}
\X : \quad y^3 = a_4x^4+a_3x^3+a_2x^2+a_1x+a_0
\end{equation}
The gap sequence is
\[ 0, 3, 4, 6, 7, 8, 9, 10 ,  \ldots \]
with matrix  $B_{3, 4}$ being 
\[
B_{3, 4}= \begin{bmatrix}
1 & x & x^2 & x^ 3 \\
y & yx & 0 & 0 \\
y^2 & 0 & 0 & 0 \\
\end{bmatrix}
\]
and the ordered basis $\B$ is 
\begin{equation}\label{Picard-basis}
1, x, y, x^2, xy, y^2, x^3, 
\end{equation}
\subsubsection{Inverting divisors}
Let us use the previous section to obtain explicit divisor inversion formula for curves of the form 
\begin{equation}\label{X}
\X : \quad y^3 = a_4x^4+a_3x^3+a_2x^2+a_1x+a_0
\end{equation}
In this case $g=3$ and the first monomials are: $1,x,y,x^2,yx,y^2,x^3.$ But to invert a divisor of degree $3$ ,$1,x,x^2,y$ will suffice. Similar to lemma \ref{curve2eq} conclude that the equation of the curve $\Y$ is: 
\begin{equation}
y-(y_1l_1(x)+y_2l_2(x)+y_3l_3(x))=0
\end{equation}
Using equation \cref{X} for our Picard curve conclude, that the equation satisfied by the intersection of $\X \cap Y$ is: 
\begin{equation}
a_4x^4+a_3x^3+a_2x^2+a_1x+a_0=(y_1l_1(x)+y_2l_2(x)+y_3l_3(x))^3
\end{equation} 
Expanding the left hand side we get: 
\begin{equation}
a_4x^4+a_3x^3+a_2x^2+a_1x+a_0=\sum_{i=1}^3y_j^3l_j^3+\sum_{\substack{i,j,k=1 \\ i\neq j\neq k}}^36l_i(x)l_j(x)l_k(x)
+3l_i^2(x)l_j(x)
\end{equation} 
Using the polynomial : $q(x)=\prod_{i=1}^3(x-x_i)$ we rewrite the last equation as: 
\begin{multline}
0=\sum_{i=1}^3y_j^3l_j^3-(a_4x^4+a_3x^3+a_2x^2+a_1x+a_0)\\+q(x)^3\left[\sum_{\substack{i,j,k=1 \\ i\neq j\neq k}}^36\frac{y_jy_ky_i}{q'(x_i)q'(x_j)q'(x_k)(x-x_i)(x-x_k)(x-x_j)}
+3\frac{y_i^2y_j^2}{q'(x_i)^2q'(x_j)(x-x_i)^2(x-x_j)}\right]
\end{multline}
Similar to $g=2$ case we have the following lemma: 
\begin{lem}
$\sum_{i=0}^4a_ix^i-\sum_{j=1}^3y_j^3l_j^3$ is divisible by $q(x)$
\end{lem}
\begin{proof}
Let $u(x)=\sum_{i=0}^4a_ix^i-\sum_{j=1}^3y_j^3l_j^3$ we need to show that $u(x_l)=0.$ Recall that $l_j(x_l)=\delta_{lj},$ Thus $l_j(x_l)^3=\delta_{jl}.$ hence: 
\begin{equation}
u(x_l)=\sum_{i=0}^4a_ix_l^i-\sum_{j=1}^3y_j^3l_j(x_l)^3=y_l^3-\sum_{j=1}^3y_j^3l_j(x_l)^3=y_l^3-\sum_{j=1}^3y_j^3\delta_{lj}=0
\end{equation}
\end{proof}
Now define $h_1(x)=\frac{u(x)}{q(x)}$ ( the coefficients of $h_1(x)$ are expressions of coefficients of $q(x),h(x)$
We have the following corollary: 
\begin{cor}
The $x$ coordinates for the inversion satisfying the equation: 
\begin{multline}
0=h_1(x)\\+q(x)^2\left[\sum_{\substack{i,j,k=1 \\ i\neq j\neq k}}^36\frac{y_jy_ky_i}{q'(x_i)q'(x_j)q'(x_k)(x-x_i)(x-x_k)(x-x_j)}
+3\frac{y_i^2y_j^2}{q'(x_i)^2q'(x_j)(x-x_i)^2(x-x_j)}\right]
\end{multline}
\end{cor}

\section{Final remarks}

All our method will work over any field $k$  of $\ch k = p>0$ as long as $(p, n ) =1$, in other words the superelliptic projection $\pi : \X \to \P^1$ is a tame cover.      It is still an open question how to extend this method to  general curves.  The main difficulty is to determine the first $2g$ elements of the  basis of $k(\X)/k$. ordered according to their order at a fixed point $P\in \X$.   As far as we are aware this is not known at the moment. 

We hope that other researchers will work out explicitly all the addition formulas for low genus superelliptic curves similarly to those of hyperelliptic curves.  It seems as that would be straightforward computational exercises, even though not necessarily easy.   Precise formulas for the addition on non-hyperelliptic curves could help our understanding of  division polynomials or torsion points for all superelliptic curves.  We hope that this note will encourage further research in those directions. 

\bigskip

\noindent \textbf{Acknowledgment} 
Both authors want to thank Julia Bernatska for helpful discussions  and for pointing out  the paper \cite{Leykin}. 

\bibliographystyle{amsplain} 

\bibliography{paper-1}{}

\end{document}

%% file: macro.tex
\usepackage{mathtools}
\usepackage{amssymb, hyperref}
\usepackage{amsmath,amssymb,amsbsy,amsfonts, latexsym,amsopn,amstext, amsxtra,euscript,amscd}
\usepackage{amsthm}
\usepackage{amsmath}

\usepackage{cite}
\usepackage[capitalise]{cleveref}

\crefname{thm}{Thm.}{}
\crefname{prop}{Prop.}{}
\crefname{lem}{Lem.}{}
\crefname{cor}{Cor.}{}

\usepackage{amsmath, amsthm}
\usepackage{amssymb}
\usepackage{latexsym} 
\usepackage{layout} 
\usepackage{algorithm}
\usepackage{algorithmic}
\usepackage{url}
\usepackage{multicol}
\usepackage{rotating}
\usepackage{booktabs}
\usepackage{multirow}

 \usepackage[alphabetic, msc-links, backrefs]{amsrefs}

\hypersetup{
  colorlinks   = true, 
  urlcolor     = blue, 
  linkcolor    = blue, 
  citecolor   = red 
}

\newtheorem{thm}{Theorem}
\newtheorem*{thm*}{Theorem}

\newtheorem{prop}{Proposition}
\newtheorem{lem}{Lemma}
\newtheorem{cor}{Corollary}
\newtheorem{rem}{Remark}
\newtheorem{exa}{Example}

\theoremstyle{definition}

\theoremstyle{remark}

\def\Z{\mathbb Z}

\def\P{\mathbb P}

\def\X{\mathcal{C}}
\def\X{\mathcal{X}}

\DeclareMathOperator\Aut{Aut }

\DeclareMathOperator\Jac{Jac }

\DeclareMathOperator\mult{mult }
\DeclareMathOperator\ord{ord }

\DeclareMathOperator\Pic{Pic}
\DeclareMathOperator\Div{Div}

\DeclareMathOperator\Sym{Symm}

\DeclareMathOperator\ch{char }

\def\<{\langle}
\def\>{\rangle}

\def\a{\alpha}

\def\Y{\mathcal Y}

\def\s{\sigma}

\def\B{\mathcal B}

\newcommand{\comment}[1]{\textcolor{red}{#1}}

%% file: Kopeliovich-Shaska.bbl
\begin{bibdiv}
\begin{biblist}

\bib{baker}{book}{
      author={Baker, H.~F.},
       title={Abelian functions},
      series={Cambridge Mathematical Library},
   publisher={Cambridge University Press, Cambridge},
        date={1995},
        ISBN={0-521-49877-5},
        note={Abel's theorem and the allied theory of theta functions, Reprint
  of the 1897 original, With a foreword by Igor Krichever},
      review={\MR{1386644}},
}

\bib{Leykin}{article}{
      author={Buchstaber, V.~M.},
      author={Le\u{\i}kin, D.~V.},
       title={Addition laws on {J}acobians of plane algebraic curves},
        date={2005},
        ISSN={0371-9685},
     journal={Tr. Mat. Inst. Steklova},
      volume={251},
      number={Neline\u{\i}n. Din.},
       pages={54\ndash 126},
      review={\MR{2234377}},
}

\bib{buchstaber}{article}{
      author={Buchstaber, Victor},
      author={Leykin, Dmitry},
       title={Hyperelliptic addition law},
        date={2005},
        ISSN={1402-9251},
     journal={J. Nonlinear Math. Phys.},
      volume={12},
      number={suppl. 1},
       pages={106\ndash 123},
         url={https://doi.org/10.2991/jnmp.2005.12.s1.10},
      review={\MR{2117174}},
}

\bib{Brill-80}{article}{
      author={Brill, A.},
       title={Ueber das {A}dditionstheorem und das {U}mkehrproblem der
  elliptischen {F}unctionen},
        date={1880},
        ISSN={0025-5831},
     journal={Math. Ann.},
      volume={17},
      number={1},
       pages={87\ndash 102},
         url={https://doi.org/10.1007/BF01444124},
      review={\MR{1510055}},
}

\bib{cantor-1}{article}{
      author={Cantor, David~G.},
       title={Computing in the {J}acobian of a hyperelliptic curve},
        date={1987},
        ISSN={0025-5718},
     journal={Math. Comp.},
      volume={48},
      number={177},
       pages={95\ndash 101},
         url={https://doi.org/10.2307/2007876},
      review={\MR{866101}},
}

\bib{coolidge}{book}{
      author={Coolidge, Julian~Lowell},
       title={A treatise on algebraic plane curves},
   publisher={Dover Publications, Inc., New York},
        date={1959},
      review={\MR{0120551}},
}

\bib{del-centina}{article}{
      author={Del~Centina, Andrea},
       title={Weierstrass points and their impact in the study of algebraic
  curves: a historical account from the ``{L}\"{u}ckensatz'' to the 1970s},
        date={2008},
        ISSN={0430-3202},
     journal={Ann. Univ. Ferrara Sez. VII Sci. Mat.},
      volume={54},
      number={1},
       pages={37\ndash 59},
         url={https://doi.org/10.1007/s11565-008-0037-1},
      review={\MR{2403373}},
}

\bib{frey-shaska}{incollection}{
      author={Frey, Gerhard},
      author={Shaska, Tony},
       title={Curves, {J}acobians, and cryptography},
        date={2019},
   booktitle={Algebraic curves and their applications},
      series={Contemp. Math.},
      volume={724},
   publisher={Amer. Math. Soc., Providence, RI},
       pages={279\ndash 344},
         url={https://doi.org/10.1090/conm/724/14596},
      review={\MR{3916746}},
}

\bib{fulton}{book}{
      author={Fulton, William},
       title={Algebraic curves},
      series={Advanced Book Classics},
   publisher={Addison-Wesley Publishing Company, Advanced Book Program, Redwood
  City, CA},
        date={1989},
        ISBN={0-201-51010-3},
        note={An introduction to algebraic geometry, Notes written with the
  collaboration of Richard Weiss, Reprint of 1969 original},
      review={\MR{1042981}},
}

\bib{fulton-1}{book}{
      author={Fulton, William},
       title={Young tableaux},
      series={London Mathematical Society Student Texts},
   publisher={Cambridge University Press, Cambridge},
        date={1997},
      volume={35},
        ISBN={0-521-56144-2; 0-521-56724-6},
        note={With applications to representation theory and geometry},
      review={\MR{1464693}},
}

\bib{hancock}{misc}{
      author={{Hancock}, Harris},
       title={{Eine Form des Additionstheorems f\"ur hyperelliptische
  Functionen erster Ordnung.}},
    language={German},
         how={{Berlin. 41 S. \(4^\circ\) (1894).}},
        date={1894},
}

\bib{jacobi-1846}{article}{
      author={Jacobi, C. G.~J.},
       title={\"{U}ber eine neue {M}ethode zur {I}ntegration der
  hyperelliptischen {D}ifferentialgleichungen und \"{u}ber die rationale {F}orm
  ihrer vollst\"{a}ndigen algebraischen {I}ntegralgleichungen},
        date={1846},
        ISSN={0075-4102},
     journal={J. Reine Angew. Math.},
      volume={32},
       pages={220\ndash 226},
         url={https://doi.org/10.1515/crll.1846.32.220},
      review={\MR{1578529}},
}

\bib{lazarsfeld}{article}{
      author={Lazarsfeld, Robert},
       title={Excess intersection of divisors},
        date={1981},
        ISSN={0010-437X},
     journal={Compositio Math.},
      volume={43},
      number={3},
       pages={281\ndash 296},
         url={http://www.numdam.org/item?id=CM_1981__43_3_281_0},
      review={\MR{632430}},
}

\bib{Leitenberger}{article}{
      author={Leitenberger, Frank},
       title={About the group law for the {J}acobi variety of a hyperelliptic
  curve},
        date={2005},
        ISSN={0138-4821},
     journal={Beitr\"age Algebra Geom.},
      volume={46},
      number={1},
       pages={125\ndash 130},
      review={\MR{2146447}},
}

\bib{Mum-1}{book}{
      author={Mumford, David},
       title={Curves and their {J}acobians},
   publisher={The University of Michigan Press, Ann Arbor, Mich.},
        date={1975},
      review={\MR{0419430}},
}

\bib{tata-2}{book}{
      author={Mumford, David},
       title={Tata lectures on theta. {II}},
      series={Progress in Mathematics},
   publisher={Birkh\"{a}user Boston, Inc., Boston, MA},
        date={1984},
      volume={43},
        ISBN={0-8176-3110-0},
         url={https://doi.org/10.1007/978-0-8176-4578-6},
        note={Jacobian theta functions and differential equations, With the
  collaboration of C. Musili, M. Nori, E. Previato, M. Stillman and H.
  Umemura},
      review={\MR{742776}},
}

\bib{Noether-84}{article}{
      author={Noether, M.},
       title={Rationale {A}usf\"{u}hrung der {O}perationen in der {T}heorie der
  algebraischen {F}unctionen},
        date={1884},
        ISSN={0025-5831},
     journal={Math. Ann.},
      volume={23},
      number={3},
       pages={311\ndash 358},
         url={https://doi.org/10.1007/BF01446394},
      review={\MR{1510258}},
}

\bib{salmon}{book}{
      author={Salmon, George},
       title={A treatise on the higher plane curves: intended as a sequel to
  ``{A} treatise on conic sections''},
      series={3rd ed},
   publisher={Chelsea Publishing Co., New York},
        date={1879},
      review={\MR{0115124}},
}

\bib{severi}{article}{
      author={Severi, Francesco},
       title={Il concetto generale di molteplicit\`a delle soluzioni pei
  sistemi di equazioni algebriche e la teoria dell'eliminazione},
        date={1947},
        ISSN={0003-4622},
     journal={Ann. Mat. Pura Appl. (4)},
      volume={26},
       pages={221\ndash 270},
         url={https://doi.org/10.1007/BF02415380},
      review={\MR{27552}},
}

\bib{sylvestre}{article}{
      author={Sylvester, James~Joseph},
       title={"question 7382"},
        date={1884},
     journal={Mathematical Questions from the Educational Times. 41: 21.},
}

\bib{towse}{article}{
      author={Towse, Christopher},
       title={Weierstrass points on cyclic covers of the projective line.},
        date={1996},
     journal={Trans. Amer. Math. Soc.},
      volume={348},
      number={8},
       pages={3355\ndash 3378},
}

\bib{W}{book}{
      author={Weierstrass, Karl},
       title={Mathematische {W}erke. {III}. {A}bhandlungen 3},
   publisher={Georg Olms Verlagsbuchhandlung, Hildesheim; Johnson Reprint
  Corp., New York},
        date={1967},
      review={\MR{0218193}},
}

\end{biblist}
\end{bibdiv}
